\documentclass{article}

\usepackage{graphics}
\usepackage{bm}
\usepackage{amssymb}
\usepackage{amsmath}
\usepackage{amsthm}
\usepackage{mathrsfs}
\usepackage[backend=bibtex,maxbibnames=10, giveninits=true,doi=false,url=false]{biblatex}
\usepackage{multirow}
\usepackage{pb-diagram}
\usepackage{comment}
\usepackage{enumerate}
\usepackage{authblk}

  \def\bmw{{\bm w}}
  
\def\cA{\mathcal{A}}  
  \def\cF{\mathcal{F}}
\def\cG{{\mathcal G}}   
\def\cJ{{\mathcal J}} 
    
\def\cP{{\mathcal P}} \def\cR{{\mathcal R}}

\def\bbN{\mathbb{N}}
\def\bbZ{\mathbb{Z}}

\def\-{\mathchar`-}

\DeclareMathOperator{\FS}{FS}

\newtheorem{thm}{Theorem}[section]
\newtheorem*{thm*}{Theorem}
\newtheorem{proposition}[thm]{Proposition}

\newtheorem*{fact*}{Fact}

\theoremstyle{definition}

\newtheorem*{defin*}{Definition}

\newtheorem*{observation*}{Observation}

\theoremstyle{remark}
\newtheorem{remark}[thm]{Remark}

\newtheorem*{open*}{Open Problem}

\newtheorem*{prob*}{Problem}

\newtheorem*{conjecture*}{Conjecture}

\AtBeginDvi{}

\numberwithin{equation}{section}

\def\fixed{\mathrm{fixed}}
\def\residual{\mathrm{residual}}

\DeclareMathOperator{\PF}{PF}

\date{}

\title{On the construction of a family of sets of positive integers closed under taking subsets
}

\author{Shoichi Kamada\thanks{Supported by JSPS KAKENHI Grant Number 19J00126, Japan and by JST CREST Grant Number JPMJCR2113, Japan.}}

\affil{Institute of Systems and Information Engineering, \\
University of Tsukuba\\
1-1-1, Tennodai, Tsukuba, Ibaraki, Japan\\
kamada.shoichi.ft@u.tsukuba.ac.jp}

\begin{document}

\maketitle

\begin{abstract}
In the several contexts such as combinatorial number theory, families of sets of positive integers closed under taking subsets have been investigated.
Then it is sometimes useful to give bijections between 
the set of the one-sided infinite sequences on the alphabet set $\{0,1\}$ and such a family of sets.
The most typical example is the family 
of sum-free sets.
Although such a kind of families covers a large class of families of sets, 
there are only a few considerations on bijections 
for the case where the sum-free property is replaced 
by another property.

In this paper, 
we explicitly give a bijection and its inverse 
between the set of one-sided infinite sequences on the alphabet set $\{0,1\}$
and a family of sets which may be contained in a class of families closed under taking subsets. 
Moreover, we show that some extremal property 
in a particular family of sets is characterized 
by a discrete dynamical system 
based on this kind of bijections. 

\end{abstract}


\section{Introduction}

Let $\cA$ be a family of sets of positive integers 
and let $\{0,1\}^{\bbN}$ be the set of one-sided infinite sequences on the alphabet set $\{0,1\}$, where $\bbN:=\{1,2,\ldots\}$. 
Throughout this paper, we are interested in the following condition for $\cA$.
\begin{itemize}
    \item[(I)] If $A\in \cA$ and $B\subseteq A$, then $B\in\cA$,  \label{item:ind}
\end{itemize}
i.e. $\cA$ is closed under taking subsets.
Especially, we will discuss on bijections
between $\{0,1\}^{\bbN}$ and $\cA$, 
where $\cA$ may be contained in a class of families satisfying~(I).
With a bijection from $\{0,1\}^{\bbN}$ to $\cA$, 
we shall say that the bijection constructs a family $\cA$ 
if it maps all elements in $\{0,1\}^{\bbN}$ to $\cA$.
In Cameron and Erd\"{o}s's paper \cite{CameronErdos1990},
several examples for families satisfying (I) are found and 
include mainly three parts: additive properties, multiplicative properties and the others.

For several families of sets with additive properties such as sum-free sets and $3$-AP free sets, it is sufficient for the condition (I) 
to avoid a solution of a single linear equation over~$\bbZ$. 
For some general additive properties, 
if a single linear equation is replaced by a system of linear equations, 
then the family of sets satisfying the condition (I) are obtained. 
Notice that the avoidance of a solution of a system of linear equations means the avoidance of a solution of some single linear equation.
As another direction, one can consider the avoidance of a solution of all single linear equations in a finite set of equations. 
Then the obtained family of sets also satisfies the condition (I).

For one multiplicative property, we can consider the pairwise coprimality. 
Then the obtained family of sets satisfies the condition (I).


As a typical example to investigate a bijection that constructs a family of sets, 
the family of sum-free sets are investigated rather than the other properties. 
In this case, it plays some role in several literatures such as 
\cite{cameron1985cyclic, cameron1987random, cameron1987portrait, luczak1995sum, calkin1996conditions} 
to give at least one of a bijection from $\{0,1\}^{\bbN}$ to $\cA$ and its inverse, explicitly. 
%
In particular, 
the Cayley graph on the group~$\bbZ$ whose generating set is a sum-free set 
is related to its cyclic automorphisms \cite{cameron1985cyclic}.



We have the following three purposes in this paper. 
The first one aims to give a method for the construction of a family $\cA$ satisfying (I).
The second one is that a bijection that constructs the family of sets satisfying the pairwise coprimality 
includes the sieve of Eratosthenes.
The last one is to introduce and investigate a new notion of ultimately completeness in our framework 
for a special class of families $\cA$ through the inverse of a bijection that constructs $\cA$.


The organization of this paper is as follows. 
In Section \ref{sec:previous}, we review the previous works due to Cameron \cite{cameron1985cyclic} and due to Calkin and Finch \cite{calkin1996conditions} on bijections between the set of one-sided infinite sequences $\{0,1\}^{\bbN}$
and the family of sum-free sets on positive integers.
In Section \ref{sec:GenFramework}, we introduce 
a general framework to describe bijections between $\{0,1\}^{\bbN}$ 
and a family of sets of positive integers that may be contained in a class of families under taking subsets.   
In Section~\ref{sec:Example}, 
we describe examples of families closed under taking subsets and their bijections. 
Some extremal property appears from a basic viewpoint of dynamical systems that  from some parametrization of bijections in Subsection \ref{subsec:norm_k}.
Moreover, it is shown that a bijection that constructs the family of sets satisfying the pairwise coprimality in our framework includes the sieve of Eratosthenes 
in Subsection~\ref{subsec:coprime}. 

\section{Preliminaries}
\label{sec:preliminaries}

Let $\{0,1\}^{\bbN}$ denote 
the set of one-sided infinite sequences
on the alphabet set $\{0,1\}$. 
In other words, 
let $\{0,1\}^{\bbN}:=\{\sigma=(\sigma_{a})_{a\in\bbN}\colon \sigma_{a}\in\{0,1\}\}$.
Notice that we can describe each of elements in $\{0,1\}^{\bbN}$ 
not only as a sequence $\sigma=(\sigma_{a})_{a\in\bbN}$ but also as a function $\sigma\colon a\mapsto \sigma_{a}$.

Let $2^{\bbN}$ denote the power set of $\bbN$. 
Then by abuse of notation, 
we denote by $A$ the characteristic function 
of $A\in 2^{\bbN}$, 
i.e. 
\[
{A}(a):=\left\{
        \begin{array}{cl}
            1 & \text{if } a\in A, \\
            0 & \text{if } a\not\in A. 
        \end{array}
\right.
\]

\noindent
To justify this, we may check that 
the power set $2^{\bbN}$ 
has one-to-one correspondence to the set $\{0,1\}^{\bbN}$ of one-sided infinite sequences.

Here, we define a metric $\rho$ on $2^{\bbN}$ by
\begin{equation}\label{}
\rho(A,B)=\left\{
        \begin{array}{cl}
        2^{-N_{0}+1} & \text{if } A\neq B, \\
        0 & \text{if } A=B,
        \end{array}
        \right.
\end{equation}
where $N_{0}$ is the smallest integer such that
\begin{equation*}
A\cap [N_{0}] \neq B\cap [N_{0}].
\end{equation*}

\noindent
i.e.
\begin{equation}\label{eq:order_in_the_metric}
N_{0}:=\min\{a\in\bbN\colon A\cap [a] \neq B\cap [a]\}.
\end{equation}

\noindent
For the characteristic functions $A$ and $B$, 
the definition of $N_{0}$ in \eqref{eq:order_in_the_metric} is replaced by
\[
N_{0}:=\min\{a\in\bbN\colon A(a)\neq B(a)\}.
\]
Throughout this paper, we use this metric to describe some topological properties related to bijections and families of sets.



\section{Sum-Free Sets}
\label{sec:previous}
In this section,  we review the works due to Cameron \cite{cameron1985cyclic} and due to Calkin and Finch \cite{calkin1996conditions}.
A set $A\subseteq \bbN$ is sum-free 
if it holds that
\begin{equation}\label{eq:sumfree}
A \cap (A+A) =\emptyset,
\end{equation}
where $A+A=\{a+b\colon a,b\in A\}$.
Clearly, the condition \eqref{eq:sumfree} is useful for 
a forbidden structure that 
taking sum is not closed in $A$. In this section, let $\cA\subseteq 2^{\bbN}$ be the family of sum-free sets. 

\paragraph{From $\{0,1\}^{\bbN}$ to $\cA$}
Here we give a bijection from $\{0,1\}^{\bbN}$ to $\cA$, say $\cF$. 
For an arbitrary $\sigma\in\{0,1\}^{\bbN}$, we can give a set $A\in\cA$ with an auxiliary set $B\in2^{\bbN}$
given by the following recursive procedure.
First, we put $A_{0}=\emptyset$ and $B_{0}=\emptyset$.
For each $n=1,2,\ldots$, if $a_{n}$ is the smallest integer which is not in 
\[
A_{n-1}\cup(A_{n-1}+A_{n-1})\cup B_{n-1}, 
\]
then 
\begin{gather*}
A_{n}=\left\{
		\begin{array}{ccc}
		A_{n-1}\cup \{a_{n}\} & \text{if} \ \sigma_{n}=1, \\
		A_{n-1}				& \text{if} \ \sigma_{n}=0,
		\end{array}
\right. 
\intertext{and}
B_{n}=\left\{
		\begin{array}{ccc}
		B_{n-1}				& \text{if} \ \sigma_{n}=1,\\
		B_{n-1}\cup \{a_{n}\} & \text{if} \ \sigma_{n}=0. 
		\end{array}
\right.
\end{gather*}

\noindent
Putting $A=\bigcup_{n=1}^{\infty}A_{n}$ and $B=\bigcup_{n=1}^{\infty}B_{n}$, 
we obtain $A=\cF(\sigma)$.

\paragraph{From $\cA$ to $\{0,1\}^{\bbN}$}

Next, we give the inverse of $F$, say $\cG$.
Let $\cG\colon \cA \to \{0,1\}^\bbN; A\mapsto \sigma$ 
denote a mapping given by the following procedure.
First, we calculate the following ternary sequence 
$\sigma'=\{\sigma_{a}'\}_{a\in\bbN}\in\{0,1,*\}^{\bbN}$.

\[
\sigma'_{a}=\left\{
\begin{array}{cl}
    1 & \text{if}\ a \in A,  \\
    * & \text{if}\ a \in A+A, \\
    0 & \text{otherwise}.
\end{array}\right.
\]

\noindent
Deleting $*$'s in $\sigma'$, we obtain $\sigma=\cG(A)\in\{0,1\}^{\bbN}$.


Here, we review the ultimately completeness of sum-free sets for short.  
A~sum-free set $A\subseteq \bbN$ is ultimately complete 
if 
for any sufficiently large $c\in \bbN\setminus A$, there exists $a,b\in A$ such that $c=a+b$, i.e. $c\in A+A$. 
\begin{proposition}[\cite{calkin1996conditions}]
A sum-free set $A$ is ultimately complete if and only
if the sequence $\cG(A)\in\{0,1\}^{\bbN}$ contains only finitely many zeros.    
\end{proposition}
In the next section, we also consider a generalization of ultimately completeness. 

\section{General Framework}\label{sec:GenFramework}

In this section, we consider a generalization of the condition \eqref{eq:sumfree}. 
Intutively, we replace $A+A$ in \eqref{eq:sumfree} 
by another operation. 
Let $\cJ$ denote an operator on the power set $2^{\bbN}$
to describe a forbidden structure, where $\cJ(\emptyset)=\emptyset$.
Then in order to describe a forbidden structure 
with respect to $\cJ$, 
one may directly consider the following condition. 
\[
A\cap\cJ(A)=\emptyset.
\]
However, this fails since there is some possibility such that
\[
A\subseteq \cJ(A)
\]
for any $A\in2^{\bbN}$.
To describe a forbidden structure,  
we need to deal with this case correctly.

Now, we regard a set $A\subseteq \bbN$ 
as an increasing sequence \mbox{$A=\{a_{1}<a_{2}<\cdots\}$.}
For each $i=1,2,\ldots$, we introduce an operator $\cJ_{i}$ on $2^{\bbN}$
defined by 
\begin{equation}\label{eq:cJi}
\cJ_{i}(A):=\{a\in \bbN \colon a_{i}<a<a_{i+1}~\text{and}~a\in \cJ(\{a_{1},\ldots,a_{i}\})\},
\end{equation}
where if $A$ is finite, then for $i\geq |A|$, it is defined by
\[
\cJ_{i}(A):=\{a\in \bbN \colon a_{|A|}<a<\infty~\text{and}~a\in \cJ(\{a_{1},\ldots,a_{i}\})\}.
\]
The operators $\cJ_{i}$'s are useful to give our bijection.
Corresponding to an operator~$\cJ$, we also introduce an operator~$\cJ^{+}$ defined by
\begin{equation}\label{eq:cJ+}
\cJ^{+}(A):=\bigcup_{i=1}^{\infty}\cJ_{i}^{+}(A),
\end{equation}
where
\begin{equation}\label{eq:cJi+}
\cJ_{i}^{+}(A):=\{a\in \bbN \colon a_{i}<a<\infty ~\text{and}~a\in \cJ(\{a_{1},\ldots,a_{i}\})\}.
\end{equation}
Then we can give the following condition.
\begin{equation}\label{eq:forbidden}
A\cap \cJ^{+}(A)=\emptyset.
\end{equation}
It is easy to see that 
the condition \eqref{eq:forbidden} generalizes the condition \eqref{eq:sumfree}. 
Indeed, in the case of $\cJ(A)=A+A$, it holds that 
\[
\cJ^{+}=\cJ
\] 
and 
\[
\cJ=\bigcup_{i=1}^{\infty}\cJ_{i}
\]
as the equality of operators. 
In general, it always holds that
\begin{equation*}
A\cap \left(\bigcup_{i=1}^{\infty}\cJ_{i}(A)\right)=\emptyset
\end{equation*}
for any $A\in2^{\bbN}$. 
Hence, the condition \eqref{eq:forbidden} can be considered 
as defining a forbidden structure with respect to~$\cJ$.


We will give examples of $\cJ$ in Section \ref{sec:Example}. 
In this section. we shall suppose that 
a subfamily $\cA\subseteq 2^{\bbN}$ satisfies (I).

\paragraph{From $\{0,1\}^{\bbN}$ to $\cA$}

Let  
$\cF\colon \{0,1\}^{\bbN}\to\cA; 
\sigma\mapsto A
$ 
denote a mapping given by the following recursive procedure.
First, we put $A_{0}=\emptyset$, $B_{0}=\emptyset$.
For each $n=1,2,\ldots,$, if $a_{n}$ is the smallest positive integer which is not in 
\[
A_{n-1}\cup \cJ(A_{n-1})\cup B_{n-1},
\]
then 
\begin{gather*}
A_{n}=\left\{
		\begin{array}{ccc}
		A_{n-1}\cup \{a_{n}\} & \text{if} \ \sigma_{n}=1, \\
		A_{n-1}				& \text{if} \ \sigma_{n}=0,
		\end{array}
\right. 
\intertext{and}
B_{n}=\left\{
		\begin{array}{ccc}
		B_{n-1}				& \text{if} \ \sigma_{n}=1,\\
		B_{n-1}\cup \{a_{n}\} & \text{if} \ \sigma_{n}=0. 
		\end{array}
\right.
\end{gather*}

\noindent
Putting $A=\bigcup_{n=1}^{\infty}A_{n}$ and 
$B=\bigcup_{n=1}^{\infty}B_{n}$, 
we obtain $A=\cF(\sigma)$.
\paragraph{From $\cA$ to $\{0,1\}^{\bbN}$}
\label{subsec:generalcG}

Here, 
we consider the inverse direction of $\cF\colon\{0,1\}^{\bbN}\to\cA$. 
Although 
the inverse of $\cF\colon\{0,1\}^{\bbN}\to\cA$ is formally given by 
$\cG=\cF^{-1}$ as a mapping from $\cA$ to $\{0,1\}^\bbN$ 
in Section \ref{sec:previous}, 
we shall give a mapping $\cG\colon2^{\bbN}\to\{0,1\}^{\bbN}$ such that the restriction of domain of $\cG$ to $\cA$ coincides $\cF^{-1}$
for convinience.

Let $\cG\colon 2^{\bbN} \to \{0,1\}^\bbN; 
A\mapsto \sigma$ 
denote a mapping constructed by the following procedure.

First, we calculate the following ternary sequence 
$\sigma'=\{\sigma_{a}'\}_{a\in\bbN}\in\{0,1,*\}^{\bbN}$.

\[
\sigma'_{a}=\left\{
\begin{array}{cl}
    1 & \text{if}\ a \in A,  \\
    * 
    & \text{if}\ a \in \cJ_{i}(A)\ \text{for some $i$}, \\
    0 & \text{otherwise}.
\end{array}\right.
\]


\noindent
Deleting $*$'s in $\sigma'$, we obtain $\sigma=\cG(A)\in\{0,1\}^{\bbN}$.

Notice that \eqref{eq:forbidden} always holds.
Then the following theorems are direct consequences.


\begin{thm}
Given the relative topology on $\cA$,
the following statements hold.
\begin{itemize}
\item A mapping $\cF\colon\{0,1\}^{\bbN}\to\cA $ 
is a continuous bijection.
\item A mapping ${\cG}\colon 2^{\bbN} \to \{0,1\}^\bbN$ is a continuous surjection.
\item The restriction of the domain of $\cG$ to $\cA$ 
is a continuous bijection. Consequently, $\cF$ is a homeomorphism.
\end{itemize}
\end{thm}

From this theorem, we see that $\cA$ is a Cantor space.


With operators $\cJ$ and $\cJ_{i}$'s, we can introduce a kind of ultimately completeness.  
A set $A\in\cA$ is ultimately complete with respect to $\cJ$
if for any $a\in \bbN\setminus A$, it holds that $a\in \bigcup_{i=1}^{\infty}\cJ_{i}(A)$.
In terms of a mapping $\cF$, a set $A=\cF(\sigma)$ for $\sigma\in\{0,1\}^{\bbN}$ is ultimately complete if $\sigma$ is of the form $\bmw\|\bm{1}_{\infty}$, a concatenation of a finite sequence $\bmw$ and the all-one infinite sequence $\bm{1}_{\infty}$.

We will see examples in the next section.

\section{Examples}\label{sec:Example}

\subsection{A finite set of linear equations}

As described in Introduction, 
the logical conjunction/disjunction for a finite set of linear equations give rise to the condition (I).
The logical conjunction for a finite set gives a system of linear equations. 
As a simple example, it can be considered that
\[
x_{1}+x_{2}=x_{3} \quad \text{and} \quad x_{2}-2x_{3}+x_{4}=0
\]
is a system of linear equations. 
Then the operator $J$ is given by
\[
J(A)=\bigcup_{i=1}^{4}\{x_{i}\colon x_{1}+x_{2}=x_{3}, \ x_{2}-2x_{3}+x_{4}=0,\ x_{j}\in A\ (j\neq i)\}
\]

The logical disjunction for a finite set is that ``and" in a system of linear equations replaced by ``or". That is, the above example is changed to 
\[
x_{1}+x_{2}=x_{3} \quad \text{or} \quad x_{2}-2x_{3}+x_{4}=0
\]
Then the operator $J$ is given by
\[
J(A)=\bigcup_{i=1}^{4}(X_{i}\cup Y_{i})
,
\]
where $X_{i}=\{x_{i}\colon x_{1}+x_{2}=x_{3},\ x_{j}\in A\ (j\neq i)\}$
and $Y_{i}=\{x_{i}\colon x_{2}-2x_{3}+x_{4}=0,\ x_{j}\in A\ (j\neq i)\}$.

We will give more discussion on the logical disjunction for a finite set in Section~\ref{subsec:norm_k}.

Here consider a finite set $L$ of linear equations with the logical conjunction/disjunction containing
\[
x-2y=0,
\]
and define the operator $\cJ$ as above. Then we give the following characterization for a mapping $\cF$ and $\cA$.
\begin{proposition}\label{prop:fixed_point}
Let $A=\{a_{1}<a_{2}<\cdots\}\in\cA$. 
Assume that
\[
\cF(A)=A,
\]
i.e. $A$ is a fixed point of $\cF$.
Then $A$ is a finite set. Moreover, $\max A<2a_{1}$. 
\end{proposition}

\begin{proof}
If $A$ is a singleton, then the statement holds. 
If $|A|\geq 2$, then the relations $a_{i}-2a_{1}=0$ does not always hold for all $i\geq 2$. This implies that 
\begin{equation}\label{eq:two_el_ex}
\cF(\{a_{1}<\cdots<2a_{1}\})=\{a_{1}<\cdots<b\},\quad b\geq 2a_{1}+1.
\end{equation}
Consequently, a necessary condition for $\cF(A)=A$ is that $A$ does not contain integers greater than or equal to $2a_{1}$. This concludes the statement.

\end{proof}

\subsection{Finite sums}
For a set $A\subseteq\bbN$, 
the set $\FS(A)$ of finite sums of $A$ is defined by
\[
\FS(A):=\{\sum_{a\in S} a\colon S\subseteq A\, \text{is nonempty and finite.}\}.
\]
A set $A\subseteq\bbN$ is 
ultimately complete if all of sufficiently large elements in $\bbN\setminus A$ are representable 
as elements in $\FS(A)$.




\begin{remark}
    There may exists an operator $\cJ$ such that a family $\cA$ constructed by the bijection $\cF$ does not satisfy the condition (I). 
    One of possibilities is the case where $\cJ=\FS$.
\end{remark}

\subsection{Norm $k$ mappings}\label{subsec:norm_k}



Let $I$ be an index set and $a,k\in\bbN$. 
To introduce an operator parameterized by $k$, 
we give the following condition for some $((y_{b})_{b\in A},w_{a})\in \bbZ^{A}\times \bbZ$:
\begin{eqnarray}\label{eq:cont_ext1}
\sum_{b\in A} y_{b}^2+w_{a}^2< k,\qquad
\sum_{b\in A} y_{b}b+w_{a}a=0, \qquad w_{a,i}\neq 0.
\end{eqnarray}
Then we define an operator $\cJ^{(k)}$ on $2^{\bbN}$ by
\[
\cJ^{(k)}(A):=\{a\in\bbN\colon \text{Some $((y_{b})_{b\in A},w_{a})\in \bbZ^{A}\times \bbZ$ 
satisfies \eqref{eq:cont_ext1}.}\}
\]
and an operator $\cJ_{i}^{(k)}$ by
\begin{equation*}
\cJ_{i}^{(k)}(A):=\{a\in \bbN\colon a_{i}<a<a_{i+1}~\text{and}~a\in \cJ^{(k)}(\{a_{1},\ldots,a_{i}\}) \}
\end{equation*}
for each $i=1,2,\ldots$. Operators $\cJ^{(k)+}$ and $\cJ_{i}^{(k)+}$ are also similar.
To parameterize mappings $\cF$ and $\cG$, 
we can give $\cF_{k}\colon 2^{\bbN}\to \cA_{k}$ and $\cG_{k}\colon \cA_{k}\to 2^{\bbN}$. 
To describe forbidden structures, we can give the following conditions for all $k$.
\begin{equation}\label{eq:norm_k_J_condition}
A \cap \cJ^{(k)+}(A)=\emptyset. 
\end{equation}
We say that a set $A\subseteq \bbN$ has norm at least $k$ if the condition \eqref{eq:norm_k_J_condition} holds. 
Thus, we define $\cA_{k}:=\{A\in 2^{\bbN}\colon \text{$A$ has norm at least $k$}\}$. 
Clearly, we have $\cA_{k}\supseteq \cA_{k+1}$ for any $k$.


Since $\cA_{k}\subseteq 2^{\bbN}$ and $\{0,1\}^{\bbN}$ is identified with $2^{\bbN}$, 
both of $\cF_{k}\colon \{0,1\}^{\bbN}\to \cA_{k}$ and $\cG_{k}\colon 2^{\bbN}\to \{0,1\}^{\bbN}$ can be regarded as discrete dynamical systems. 




Here, we introduce a kind of ultimately completeness.  Let $k$ be fixed. Then
a set $A\in\cA_{k}$ is ultimately complete 
if for all sufficiently large $a\in \bbN\setminus A$, it holds that $a\in \bigcup_{i=1}^{\infty}\cJ_{i}^{(k)}(A)$.




\begin{proposition}
Let $k\geq 7$ and $A=\{a_{1}<a_{2}<\cdots\}\in 2^{\bbN}$ be infinite. Then the following statements hold.
\begin{enumerate}[(1)]
    \item A limit $A_{\infty}:=\lim\limits_{n\to\infty}\cG_{k}^{n}(A)\in 2^{\bbN}$ always exists.
    \item There exists a partition $A_{\infty}=A_{\fixed}\cup A_{\residual}$ 
    such that 
    $\cG_{k}(A_{\fixed})=A_{\fixed}$, i.e. $A_{\fixed}$ is a fixed point of $\cG_{k}$. Moreover, we can take $A_{\fixed}$ such that $A_{\fixed}\cap\{1,\ldots,2a_{1}-1\}$ is a fixed point of $\cF_{k}$.
\end{enumerate}
\end{proposition}

\begin{proof}
    (1)
    Consider a dynamics $\{\cG_{k}^{n}(A)\}_{n=0}^{\infty}$. For convinience, put $\cG_{k}^{n}(A)=\{a_{1}^{(n)}<a_{2}^{(n)}<\cdots\}$.
    For any fixed $i$, a sequence of differences $\{a_{i+1}^{(n)}-a_{i}^{(n)}\}_{n}$ monotonically decreases as $n\to\infty$ 
    in a wide sense and appears the same difference ultimately. 
    Hence, a limit always exists.

    \noindent
    (2) From the fact of Proposition \ref{prop:fixed_point}, we take a look at elements in $\cG_{k}^{(n)}(A)$ less than $2a_{1}^{(n)}$.
    For a mapping $\cG_{k}$, we can also give a similar description of \eqref{eq:two_el_ex}, i.e. for $a\geq 3$, we obtain
    \[
    \cG_{k}(\{a,2a+1,\ldots\})=\{a,2a,\ldots\},
    \]
    where ``$\ldots$" in the above means elements other than the first and second smallest positive integers in each set.
    Indeed, the number of $*$'s between $a$ and $2a+1$ is 
    \[
    |\cJ_{1}^{(k)}(\{a,2a+1,\ldots\})|=|\{2a\}|=1.
    \]
    In this case, we have $A_{\fixed}=\{a,2a\}$ and $A_{\residual}:=A_{\infty}\setminus A_{\fixed}$ with $\min A_{\residual}\geq 2a_{1}$. 
    Consider a general $A=\{a_{1}<a_{2}<\cdots\}\in 2^{\bbN}$. 
    This means that $a_{i_{0}}<2a_{1}$ for some $i_{0}\geq 3$. 
    We can similarly take $A_{\fixed}$ of cardinality $\geq 3$ and take $\max A_{\fixed}\leq 2a_{1}$, and $A_{\residual}:=A_{\infty}\setminus A_{\fixed}$. 
    
    
\end{proof}

\begin{proposition}
    A sufficient condition for the ultimately completeness of a set $\cF_{k}(A_{\infty})$ is that
    all of the following conditions hold. 
    \begin{enumerate}[(1)]
        \item $A\in \cA_{k}$.
        \item $A\cup \{1\}\not \in \cA_{k-1}$.
        \item For some $\{y_{a}\}_{a\in A}\in \bbZ^{A}$ and some $w_{1}\in\bbZ$, 
        \begin{eqnarray*}
    \sum_{a\in A} y_{a}^2+w_{1}^2 \leq k-2,\qquad
    \sum_{a\in A} y_{a}a+w_{1}=0,\qquad w_{1}= 1. 
    \end{eqnarray*}
    \end{enumerate}

\end{proposition}

\begin{proof}
    From Conditions (1) and (2), we have $1\not\in A$. It is clear that Condition~(3) implies Condition (2), but the converse is not always true since there is some possibility that $y_{1}$ takes another non-zero integer.  
    To show the ultimately completeness of a set $\cF_{k}(A_{\infty})$, 
    it is sufficient to prove that for all sufficiently large~$i$, differences $a_{i+1}^{(n)}-a_{i}^{(n)}$ goes to $1$ as $n\to\infty$.
    To do this, we show that $a_{i+1}^{(n+1)}-a_{i}^{(n+1)}<a_{i+1}^{(n)}-a_{i}^{(n)}$.
    From the second and last conditions in (3), 
    we have
    \begin{equation}\label{eq:rep_of_1}
    1=-\sum_{a\in A}y_{a}a.
    \end{equation}
    Put $i_{0}=\max\{i\colon y_{a_{i}}\neq 0\}$. 
    When $i>i_{0}$, we add $a_{i}$ in both sides of \eqref{eq:rep_of_1}. 
    Hence, we have
    \[
    a_{i}+1\: =\: a_{i}-\sum_{j=1}^{i_{0}}y_{a_{j}}\hspace{-1pt}a_{j}.
    \]
    By introducing $w_{a_{i}}=1$, we have 
    \[
    a_{i}+1\: =\: w_{a_{i}}\hspace{-1pt}a_{i}-\sum_{j=1}^{i_{0}}y_{a_{j}}\hspace{-1pt}a_{j}.
    \]
    Also, by introducing $y_{a_{i}+1}=-1$, the above identity can be reduced to condition in \eqref{eq:cont_ext1}.
    The right hand side tells us that the deletion of $*$'s always appears until the difference $a_{i+1}^{(n)}-a_{i}^{(n)}$ for some $i$ is greater than $1$.    
\end{proof}


\subsection{A Family of Coprime Sets}\label{subsec:coprime}
A set $A\subseteq\bbN$ is coprime if the greatest common divisor of any two distinct $a,b\in A$ is $1$. In this subsection, let $\cA$ be the family of coprime subsets of $\bbN$. 
Now we consider the two mappings $\cF$ and $\cG$ as in Section \ref{sec:GenFramework}
how we construct the family of coprime sets of $\bbN$. 
Let $\cP$ denote the set of prime numbers and 
$\PF(A)$ the union of sets of positive prime factors for every $a\in A$, i.e.
\[
\PF(A):=\bigcup_{a\in A}\{p\in \cP\colon p\mid a\}=\{p\in \cP\colon \text{$p\mid a$ for some $a\in A$}\}.
\]
To consider the mapping $\cF$, 
we define an operator $\cJ$ by
\[
\cJ(A):=\{a\in \bbN \colon p\mid a\;\text{for some}\; p \in \PF(A)\}.
\]
As in \eqref{eq:cJi},
we can also define $\cJ_{i}(A)$ for a sequence $A=\{a_{1}<a_{2}<\cdots\}$.
It is easy to see that the set $\cP$ is a coprime set. 
Moreover, we can show that $\cF(0\|\bm{1}_{\infty})=\cP$, 
where $0\|\bm{1}_{\infty}=011\ldots$ is the concatenation of $0$ and the all-one infinite \mbox{sequence $\bm{1}_{\infty}$}. 
By the way, we can also show that $\cF\,(\bm{1}_{\infty})=\{1\}\cup\cP$.
Hence, we conclude that the procedure $\cF\colon 0\|\bm{1}_{\infty}\mapsto \cP$ can be regarded as 
the sieve of \mbox{Eratosthenes} 
and the procedure $\cF\colon\sigma\mapsto \cF(\sigma)$ for any other $\sigma\in \{0,1\}^{\bbN}$ 
can be viewed as a generalization of the sieve of Eratosthenes 
(cf. \cite{Sieve_Greaves_Book2001}).

 \section{Concluding Remarks}
 \label{sec:ConcludingRemarks-}

As a complementary case of the condition (I), 
we may consider $2^{\bbN}\setminus \cA$ or its subfamily. 
For a statement 
that a system of linear equations over $\bbZ$  has a solution in $A\subseteq\bbN$, 
assume that $\cR\subseteq2^{\bbN}$ is a family of sets that satisfy this statement.
Then such a kind of families $\cR$ involves Schur's theorem \cite{Schur1917}, van der Waerden's theorem \cite{van_der_Waerden_AP_1927}, and Rado's theorem \cite{rado1933}.
In this case, it is matched with Ramsey Theory and the theory of ultrafilters on $\bbN$ (e.g. see 
\cite{graham1991ramsey}).
Ultrafilters are useful for the complementary case. 
However, we think that there is none corresponding to ultrafilters for the condition (I) case. 
So we expect a new theory with respect to the condition (I).

\printbibliography

\end{document}